\documentclass[english]{amsart}
\usepackage[T1]{fontenc}
\usepackage[latin9]{inputenc}
\usepackage{color}
\usepackage{babel}
\usepackage{amstext}
\usepackage{amsthm}
\usepackage{amssymb}
\usepackage[pdfusetitle,
 bookmarks=true,bookmarksnumbered=false,bookmarksopen=false,
 breaklinks=false,pdfborder={0 0 0},pdfborderstyle={},backref=false,colorlinks=true]
 {hyperref}
\hypersetup{
 linkcolor=brown, citecolor=blue, pdfstartview={FitH}, hyperfootnotes=false, unicode=true}

\makeatletter
\numberwithin{figure}{section}
\theoremstyle{plain}
\newtheorem{thm}{\protect\theoremname}
\theoremstyle{plain}
\newtheorem{prop}[thm]{\protect\propositionname}
\theoremstyle{definition}
\newtheorem{problem}[thm]{\protect\problemname}
\theoremstyle{plain}
\newtheorem{lem}[thm]{\protect\lemmaname}
\theoremstyle{remark}
\newtheorem{rem}[thm]{\protect\remarkname}
\theoremstyle{plain}
\newtheorem*{thm*}{\protect\theoremname}
\theoremstyle{plain}
\newtheorem*{prop*}{\protect\propositionname}

\usepackage{tikz-cd}
\usepackage{bbm}

\makeatother

\providecommand{\lemmaname}{Lemma}
\providecommand{\problemname}{Problem}
\providecommand{\propositionname}{Proposition}
\providecommand{\remarkname}{Remark}
\providecommand{\theoremname}{Theorem}

\begin{document}
\title{Torsion-freeness problem for matrix groups}
\author{Ruiwen Dong}
\author{Emmanuel Rauzy }
\thanks{The authors are thankful to Giles Gardam, for raising the problem
solved in Theorem \ref{thm:torsion free problem}, and for providing
useful references. }
\maketitle

\section{Introduction}

In \cite{Malcev1940}, Malcev proved that any finitely generated linear
group has a faithful representation over a field which is a purely
transcendental extension of finite degree of the prime field. Since
such a field is always a computable field, it follows that finitely
generated linear groups have solvable word problem. (This was also
stated by Rabin in \cite{Rabin1960}.) In \cite{Malcev1956ResFin}
(translated in \cite{Malcev1983ResFin}), Malcev proved that these
groups are residually finite, and this in fact holds effectively \cite{Detinko2019}
and can be used to perform algorithmic investigations of linear groups. 

Following these seminal results, the study of computational problems for matrix groups has 
become a rich
topic, with positive results \cite{Detinko_Eick_Flannery_2011,DETINKO2013100},
negative results \cite{CASSAIGNE1999,breuillard2024undecidabilityMatrix},
and GAP and MAGMA implementations available for explicit computations. 

$ $

In the present article, we establish some undecidability results about
finitely generated linear groups. 
\begin{thm}
\label{thm:torsion free problem} The torsion-freeness problem is
undecidable for finitely generated subgroups of $\mathrm{SL}_{4}(\mathbb{Z})$. 
\end{thm}

That is to say, no algorithm can, on input a tuple of matrices in
$\mathrm{SL}_{4}(\mathbb{Z})$, decide whether
the group they generate is torsion-free. (Note that it is decidable whether a group
given in this way is torsion, i.e. finite.)

This result is obtained by using the Mihailova subgroup construction
\cite{Mihailova1968Occurrence} together with Kharlampovich's three
solvable group with unsolvable word problem \cite{Harlampovic1982}. 

The same construction yields:
\begin{thm}
\label{thm:square root problem}There exists a finitely generated
subgroup $G$ of $\mathrm{SL}_{4}(\mathbb{Z})$ with undecidable square
root problem: no algorithm can, on input an element of $G$, decide
whether or not this element is a square.
\end{thm}

Using the more advanced fiber product construction from \cite{Bridson2011},
we show: 
\begin{thm}
\label{thm:no uniform CFQ linear}Finitely generated linear groups
do not have uniformly computable marked finite quotients. 
\end{thm}

This means that there is no algorithm which can, on input some $n$
and a tuple $T$ of matrices of $\mathrm{SL}_{n}(\mathbb{Z})$, produce
a list of all finite marked groups that are quotients of the group
generated by $T$. It is known that finitely generated linear groups
are \emph{uniformly effectively residually finite} (see for instance
\cite{Detinko2019}, the name is from \cite{Rauzy2020}): there is
an algorithm which, given a tuple of matrices $T$ that generate a
group $G$, outputs a sequence $(\phi_{n}:G\rightarrow F_{n})_{n\in\mathbb{N}}$
of finite quotients of $G$ such that 
\[
\bigcap_{n\in\mathbb{N}}\ker(\phi_{n})=\{1\}.
\]

Thus there is a line to be drawn between \emph{being able to} \emph{enumerate
sufficiently many finite quotients to separate elements} \emph{of}
$G$, and \emph{being able to enumerate all the finite quotients of}
$G$. 

Finally, we record the following direct consequence of results of
\cite{Rauzy2021}:
\begin{prop}
\label{prop:fp wp linear subgroups}There exists a finitely presented
group $G$ with solvable word problem such that no algorithm can,
on input a tuple of elements of $G$, stop if and only if the group
they generate is linear, nor can an algorithm stop if and only if
the group they generate is not linear. 
\end{prop}

\subsection{Open problems}

Many questions can be asked about results similar to Theorem \ref{thm:torsion free problem}.
For instance, the following is a well known open problem:
\begin{problem}
[\cite{Detinko2019}, Section 6.2.] Is there an algorithm which,
given a tuple of matrices in $\mathrm{GL}_{n}(\mathbb{Z})$, decides
if the group they generate is free?
\end{problem}

It is natural to ask whether Theorem \ref{thm:no uniform CFQ linear}
can be strengthened to impossibility even of a non-uniform algorithm: 
\begin{problem}
Is there a finitely generated linear group with non-computable marked
finite quotients? 
\end{problem}

Finally, we ask to extend Proposition \ref{prop:fp wp linear subgroups}
to other types of descriptions of groups: 
\begin{problem}
Can there be an algorithm that stops exactly on finite presentations
of linear groups? 
\end{problem}

By the Adian-Rabin Theorem, there cannot be an algorithm that stops
exactly on finite presentations of non-linear groups. 

\section{\label{sec:Detecting-squares-in}Detecting squares in Kharlampovich's
three solvable group}

The following lemma is the key ingredient in proving Theorem \ref{thm:torsion free problem}
and Theorem \ref{thm:square root problem}.
\begin{lem}
\label{lem:squares in kharlampovich}There exists a finitely presented
group $H$ such that the following problem is undecidable:
\begin{itemize}
\item Input: an element $w$ central in $H$,
\item Output: is $w$ a square in $H$?
\end{itemize}
In fact, we can suppose that the only square that belongs to the center
of $H$ is the identity element.
\end{lem}

Mark Sapir writes the following on Mathoverflow \cite{Sapir2018Square}: 

\medskip

``There exists a group $G$ given by a finite set of generators $X$
and a finite set of relations $R$ for which there is no algorithm
to decide, given an element $g\in G$, if $g$ is a square in $G$.
One example of such a group is Kharlampovich's group from \cite{Harlampovic1982}
(one needs to take $p=2$ in her construction). 

Indeed, if an element belongs to the center of Kharlampovich's group
with $p=2$, then it is a square in the group if and only if it is
equal to $1$, and Kharlampovich proved that this property is undecidable.''

\medskip

However, the above claim is false (many central elements are squares,
this will appear clearly below), but true up to a small modification
of the construction. 
\begin{proof}
We refer ourselves throughout to the notations of \cite{Kharlampovich2017},
which gives an updated proof of \cite{Harlampovic1982}. 

The first part is to notice that indeed, in Kharlampovich's group,
the word problem is unsolvable even when restricting our attention
to central elements. This is because not only is the group $3$-solvable
(belonging to $\mathcal{A}^{3}$), but by \cite[Theorem 4.3]{Kharlampovich2017}
it also belongs to the group variety $\mathcal{ZN}_{3}\mathcal{A}$
of abelian-by-three-nilpotent-by-central groups\footnote{Note that even though the variety product is associative, it is not
so when the notation $\mathcal{Z}$ is involved. In particular here
$G\in\mathcal{Z}(\mathcal{N}_{3}\mathcal{A})$, but $G\notin(\mathcal{Z}\mathcal{N}_{3})\mathcal{A}=\mathcal{N}_{4}\mathcal{A}$}. Note that abelian-by-nilpotent groups have solvable word problem
\cite{Miller1992}. And thus the fact that Kharlampovich's group has
unsolvable word problem implies that no algorithm can, given a central
element, decide whether or not that element is the identity element. 

We now explain why Sapir's claim that central elements are not squares
is false but true up to a small modification. 

The group $G$ is associated to a Minsky machine $M$.

A configuration of $M$ is of the form
\[
q_{i}a_{1}^{n_{1}}...a_{k}^{n_{k}},\,n_{j}\in\mathbb{N},
\]
where $q_{i}$ is the state of the machine and the $a_{j}$ are its
glasses, the $j$-th glass containing exactly $n_{j}$ coins. Such
a configuration is associated a group element 
\[
x(q_{i})*a_{1}^{(n_{1})}*...*a_{k}^{(n_{k})}*A_{1}*...*A_{k},
\]
Here $x(q_{i})$ is an element that represents the state $q_{i}$,
the element $a_{j}$ represents a coin in the $j$-th glass, $A_{1}$,...,
$A_{k}$ are group elements that represent the bottom of the glasses
of the Minsky machine -thus indicating when such a glass is empty-
and $*$ is an operation obtained by applying products and conjugates
and whose exact definition we will not need.

In the construction given in \cite{Harlampovic1982,Kharlampovich2017},
relations are added so that the elements $x(q_{i})$ and $A_{i}$
have order $p$ for a chosen prime $p$. We are interested in the
case $p=2$, which is the only case where we can hope non-trivial
central elements not to be squares. 

Kharlampovich's group $G$ is a semi-direct product 
\[
T\rtimes(H_{1}^{H_{2}}\rtimes H_{2}),
\]
where $T=\langle x(q_{i})\rangle^{G}$, $H_{1}^{H_{2}}=\langle A_{i}\rangle^{H_{2}}$
and $H_{2}=\langle a_{i}\rangle$ (see \cite[p. 337]{Kharlampovich2017}).
These three groups are abelian, and $T$ and $H_{1}^{H_{2}}$ have
exponent $p=2$. The center of $G$ is contained in $T$.

It is not true that non-trivial central elements are not squares.
Indeed, for $t\in T$ and $h_{1}$ in $H_{1}$ we have 
\[
(th_{1})^{2}=th_{1}th_{1}=th_{1}t^{-1}h_{1}^{-1}=[t,h_{1}],
\]
 because $h_{1}$ and $t$ have order $2$. Many central elements
can be obtained in this way. 

However, if we modify Kharlampovich's construction in that we do not
impose that the elements $A_{1}$,... , $A_{k}$ have order $2$,
all results remain true and the following equivalence will hold:
\[
\text{For \ensuremath{a} central, (\ensuremath{a=1\iff a} is a square)}.
\]
Indeed, $G$ remains a semi-direct product 
\[
T\rtimes(H_{1}^{H_{2}}\rtimes H_{2}),
\]
the group $T$ is still abelian of exponent $2$, but now $H_{1}^{H_{2}}$
is abelian and torsion free, instead of abelian of exponent $2$.
An element of $G$ is a triple $(t,h_{1},h_{2})$, $t\in T$, $h_{1}\in H_{1}^{H_{2}}$
and $h_{2}\in H_{2}$. Then 
\[
(t,h_{1},h_{2})^{2}\in T\implies h_{2}^{2}=1\text{ because }H_{2}\text{ is torsion free}
\]
\[
(t,h_{1},1)^{2}\in T\implies h_{1}^{2}=1\text{ because }H_{1}\text{ is torsion free}
\]
\[
(t,1,1)^{2}=1\text{ because }T\text{ has exponent }2
\]
And thus the only element of $T$ that is a square is now the neutral
element. 

It is easy to check that the relations that impose $H_{1}^{H_{2}}$
to have exponent $2$ are not important in the ``undecidability of
the word problem'' part of the result. They only serve in improving
the result: a finitely presented group with unsolvable word problem
is constructed in the variety $\mathcal{A}_{2}^{2}\mathcal{A}$ of
abelian-by-abelian-of-exponent-$2$-by-abelian-of-exponent-$2$, this
is a better result than only constructing a 3-solvable group with
unsolvable word problem. 

To check the above claim, one only needs to look at the proof given
in \cite[pp. 333-343]{Kharlampovich2017}, and to see that the relations
that impose that $H_{1}$ has exponent $2$ (the relations (G1) in
the notations of \cite{Kharlampovich2017}) are only ever used in
order to prove that $G$ will belong to $\mathcal{A}_{2}^{2}\mathcal{A}$
instead of $\mathcal{A}^{3}$, this is easy enough to do. 
\end{proof}
\begin{rem}
It is easy to build a non-finitely generated abelian group with solvable
word problem where the problem of determining if a given element is
a square is undecidable. 

Take $H_{0}=\bigoplus_{i\in\mathbb{N}}\mathbb{Z}$ with $a_{i}$ as
generator of the $i$-th copy of $\mathbb{Z}$. Consider $\varphi_{i}$
a standard enumeration of all Turing machines, $(p_{i})$ the sequence
of primes, ands add the following relations:

Whenever $\varphi_{n}(n)$ halts after $t$ computation steps, add
the relation (in additive notation):
\[
a_{p_{i}}=2a_{p_{i}^{t}}
\]
to the group $H_{0}$. 

The resulting group $H_{1}$ is still isomorphic to $\bigoplus_{i\in\mathbb{N}}\mathbb{Z}$,
but the generators $(a_{i})_{i\in\mathbb{N}}$ are not a free generating
family anymore. The group $H_{1}$ still has solvable word problem
with respect to the generating family $(a_{i})_{i\in\mathbb{N}}$.
However, the problem of deciding if a given element is a square is
undecidable: $a_{p_{i}}$ is a square if and only if the $i$th Turing
Machine halts on input its own name, which is undecidable.
\begin{problem}
By Houcine's version of the Higman Embedding Theorem \cite{OuldHoucine2007},
it is possible to embed $H_{1}$ into a finitely presented group with
solvable word problem whose center will be exactly $H_{1}$. Can this
be done while preserving undecidability of the square problem for
the center? 
\end{problem}

\end{rem}

\section{\label{sec:Undecidable-square-root}Undecidable square root problem
in a linear group}

We will use the Mihailova subgroup. Let $H$ be a finitely presented
group with presentation $\pi=\langle a_{1},...,a_{k}\vert\,r_{1},...,r_{n}\rangle$.
The Mihailova subgroup associated to $\pi$ is the finitely generated
subgroup $M_{\pi}$ of $\mathbb{F}_{k}\times\mathbb{F}_{k}$ generated
by the following $(k+n)$ elements
\[
(a_{i},a_{i}),i=1..k
\]
\[
(1,r_{i}),i=1..n.
\]
It is well known that this subgroup can be characterized as
\[
(u,v)\in M_{\pi}\iff u=_{H}v.
\]

\begin{proof}
[Proof of Theorem \ref{thm:square root problem}] Note that the
Mihailova subgroup $M_{\pi}$ is indeed a subgroup of $\mathrm{SL}_{4}(\mathbb{Z})$
(independently of the presentation $\pi$). 

Let $\pi=\langle a_{1},...,a_{k}\vert\,r_{1},...,r_{n}\rangle$ be
a finite presentation of the group $H$ given in Lemma \ref{lem:squares in kharlampovich},
and consider the associated Mihailova subgroup $M_{\pi}$.

For each word $w$ over $\{a_{1},...,a_{k}\}^{\pm1}$ that defines
an element of the center of $H$, the element $(1,w^{2})$ belongs
to $M_{\pi}$, since $w^{2}=_{H}1$. 

Because free groups have the unique root property, the set of square
roots of $(1,w^{2})$ in $\mathbb{F}_{k}\times\mathbb{F}_{k}$ consists
in the singleton $\{(1,w)\}$. 

Thus $(1,w^{2})$ is a square in $M_{\pi}$ if and only if $(1,w)$
belongs to $M_{\pi}$, if and only if $w=_{H}1$, which is undecidable. 
\end{proof}

\section{\label{sec:Torsion-freeness-problem-for}Torsion-freeness problem
for finitely generated subgroups of $\mathrm{SL}_{4}(\mathbb{Z})$}
\begin{lem}
\label{lem:Reduction square problem}The following problem reduces
to the torsion-freeness problem for subgroups of $\mathrm{SL}_{4}(\mathbb{Z})$:
given a finite presentation $\pi$ for a group $H$ and an element
$w$ central in $H$, decide whether or not $w$ is a square in $H$. 
\end{lem}

\begin{proof}
As in Section \ref{sec:Undecidable-square-root}, we let $M_{\pi}$
be the Mihailova subgroup associated to a finite presentation $\pi=\langle a_{1},...,a_{k}\vert\,r_{1},...,r_{n}\rangle$. 

We will now consider subgroups of the semi-direct product $M_{\pi}\rtimes\mathbb{Z}/2$,
where $\mathbb{Z}/2$ acts by permuting factors. Note that this semi-direct
product can be realized as a subgroup of $\mathrm{SL}_{4}(\mathbb{Z})$. 

Given $w\in\mathbb{F}_{k}$ that defines a central element of $H$,
consider the group $L_{\pi,w}$ generated by the Mihailova generators
together with the element $g_{0}=(w,1)\varepsilon$, where $\varepsilon$
is the non-trivial element of $\mathbb{Z}/2$. Thus 
\[
L_{\pi,w}=\langle M_{\pi},g_{0}\rangle.
\]

We will show that $L_{\pi,w}$ contains an element of finite order
if and only if $w$ is a square in $H$. 

Suppose first that $w$ is a square in $H$. Take $b$ such that $w=_{H}b^{2}$.
Then $(bw^{-1},b^{-1})\in M_{\pi}$, because $bw^{-1}=_{H}bb^{-2}=_{H}b^{-1}$.
And thus $(bw^{-1},b^{-1})g_{0}=(b,b^{-1})\varepsilon\in L_{\pi,w}$,
and thus $L_{\pi,w}$ has torsion. 

We prove the converse implication. 

Note that $g_{0}^{2}=((w,1)\varepsilon)^{2}=(w,w)\in M_{\pi}$. Thus
an element of $L_{\pi,w}$ can be written as a product of elements
of $M_{\pi}$ and of $g_{0}^{\pm1}$ with no two consecutive occurrences
of $g_{0}^{\pm1}$. 

Note that $g_{0}(u,v)g_{0}=(wv,uw)$. If $(u,v)\in M_{\pi}$, then
$u=_{H}v$, and thus also $wu=_{H}vw$ because $w$ is central in
$H$, and so $g_{0}(u,v)g_{0}\in M_{\pi}$. 

The same goes for the products:
\[
g_{0}^{-1}(u,v)g_{0}=(v,w^{-1}uw)
\]
\[
g_{0}(u,v)g_{0}^{-1}=(wvw^{-1},w)
\]
\[
g_{0}^{-1}(u,v)g_{0}^{-1}=(vw^{-1},w^{-1}u)
\]
Thus any product of elements of elements of $M_{\pi}$ and of $g_{0}^{\pm1}$
that contains an even number of occurrences of $g_{0}^{\pm1}$ in
fact belongs to $M_{\pi}$. 

Thus any element of $L_{\pi,w}$ either belongs to $M_{\pi}$ or it
can be written as $(u_{1},v_{1})g_{0}^{\pm1}(u_{2},v_{2})$ for $(u_{1},v_{1}),(u_{2},v_{2})\in M_{\pi}$. 

Suppose that $L_{\pi,w}$ contains an element of finite order. Such
an element cannot belong to $M_{\pi}$, which is torsion-free. Thus,
up to conjugating and taking an inverse, $L_{\pi,w}$ must contain
an element of finite order of the form 
\[
(u,v)g_{0}=(uw,v)\varepsilon
\]
for some $(u,v)\in M_{\pi}$. This element has finite order if and
only if it has order exactly $2$ if and only if 
\begin{align*}
((uw,v)\varepsilon)^{2}=1 & \iff(uwv,vuw)=1\\
 & \iff uwv=1\\
 & \iff w=u^{-1}v^{-1}\\
 & \implies w=_{H}u^{-2}\text{ because }u=_{H}v
\end{align*}
Thus if there is a torsion element in $L_{\pi,w}$, $w$ is a square
in $H$. 
\end{proof}
Theorem \ref{thm:torsion free problem} follows. 
\begin{proof}
[Proof of Theorem \ref{thm:torsion free problem}] This is simply
Lemma \ref{lem:squares in kharlampovich} together with Lemma \ref{lem:Reduction square problem}.
\end{proof}

\section{Uniform computability of finite quotients}

A group $G$ with generating family $S$ has \emph{computable marked
finite quotients }if there is an algorithm which, on input a finite
group $F$ and a function $f$ from $S$ to $F$, decides whether
or not $f$ can be extended to a group homomorphism. 

This property is in fact independent of the choice of $S$ \cite{Rauzy_2021}. 

When $G$ is recursively presentable, this property is equivalent
to the existence of a computable enumeration of all its marked finite
quotients \cite{Rauzy_2021}. (In this enumeration, a finite quotient
is described by the Cayley table of the finite group together with
a function from the generators of $G$ to generators of the finite
group.)

A family $\mathcal{C}$ of groups has \emph{uniformly computable finite
quotients} if there is an algorithm which, given a group $G$ in $\mathcal{C}$
and a function from the generators of $G$ to a finite group, decides
if the function defines a group homomorphism\footnote{This definition in fact depends on both the family $\mathcal{C}$
and of a chosen way to describe the groups in $\mathcal{C}$ with
finite data. Here, for linear groups, we are concerned by groups described
by generating tuples given as matrices.}. When the groups in $\mathcal{C}$ are uniformly recursively presentable
(as are linear groups), this is equivalent to the existence of an
algorithm which, given as input a group in $\mathcal{C}$, produces
a list of its finite quotients. 

Here we prove Theorem \ref{thm:no uniform CFQ linear}: 
\begin{thm*}
[Theorem \ref{thm:no uniform CFQ linear}]\label{thm:no uniform CFQ linear-1}Finitely
generated linear groups do not have uniformly computable marked finite
quotients.
\end{thm*}
The proof is a corollary of a construction of Bridson and Wilton \cite{Bridson2011}.
\begin{thm}
[\cite{Bridson2011}, Corollary D]\label{thm:,-Corollary-D}There
exists a computable sequence of triples $(m_{n},r_{n},S_{n})$ with
$m_{n},r_{n}\in\mathbb{N}$ and $S_{n}\subseteq\mathrm{SL}_{m_{n}}(\mathbb{Z})$
a finite set such that: 
\begin{enumerate}
\item all of the groups $\Lambda_{n}:=\langle S_{n}\rangle\subseteq\mathrm{SL}_{m_{n}}(\mathbb{Z})$
are finitely presentable;
\item the set of integers $\{n\in\mathbb{N}\mid b_{1}(\Lambda_{n})=r_{n}\}$
is computably enumerable but not computable. 
\end{enumerate}
In particular, there is no algorithm that takes as input a finite
set of integer matrices that generate a finitely presentable group
and outputs a presentation for that group.
\end{thm}

With the notation of the above theorem, no algorithm can, on input
$n$, produce a finite presentation of the abelianization of $\Lambda_{n}$.

To establish Theorem \ref{thm:no uniform CFQ linear}, we will need
an additional fact that follows from the proof of Theorem \ref{thm:,-Corollary-D}
given in \cite{Bridson2011}: we can control the torsion part of the
abelianization of $\Lambda_{n}$.

\begin{prop}
\label{prop:additional prop}Additionally to what was stated in Theorem
\ref{thm:,-Corollary-D}, we can assume that there is a computable
sequence $(\pi_{n})$ of finite presentations of abelian groups such
that in case $b_{1}(\Lambda_{n})=r_{n}$, the abelianization of $\Lambda_{n}$
is $\pi_{n}$. 
\end{prop}

\begin{proof}
The construction of the sequence $(m_{n},r_{n},S_{n})$ given by Theorem
\ref{thm:,-Corollary-D} relies on the following steps:
\begin{itemize}
\item Following \cite{collins_miller_1999}, start with a computable sequence
of finite presentations defining a sequence $(Q_{n})$ of groups such
that the $Q_{n}$ is either: 
\begin{itemize}
\item trivial, or:
\item perfect, aspherical and with infinite second homology group,
\end{itemize}
and such that the set 
\[
\{n\mid Q_{n}=\{1\}\}
\]
is computably enumerable but not computable.
\item One replaces the sequence $(Q_{n})$ by the sequence $(\tilde{Q}_{n})$,
where $\tilde{Q}_{n}$ is the universal central extension of $Q_{n}$
whose presentation is obtained using \cite[Proposition 3.4]{Bridson2011}.
(Note that the trivial group is perfect, that it is its own universal
central extension, and \cite[Proposition 3.4]{Bridson2011} indeed
yields a presentation of the trivial group if the starting group was
trivial.)
\item One then applies a version of the Rips construction due to Haglund
and Wise \cite{Haglund2008} to this sequence, taking, for each group
$\tilde{Q}_{n}$, the fiber product of a short exact sequence
\[
1\rightarrow K_{n}\rightarrow\Gamma_{n}\overset{p_{n}}{\rightarrow}\tilde{Q}_{n}\rightarrow1
\]
where $\Gamma_{n}$ is virtually special, $\text{CAT}(0)$ and hyperbolic,
and $K_{n}$ is a finitely generated subgroup of $\Gamma_{n}$.
\item The group $\Lambda_{n}$ is the fiber product over this short exact
sequence, i.e. $\Lambda_{n}\le\Gamma_{n}\times\Gamma_{n}$ is given
by 
\[
(u,v)\in\Lambda_{n}\iff p_{n}(u)=p_{n}(v).
\]
\item The 1-2-3 Theorem of Baumslag, Bridson, Miller and Short \cite{Baumslag2000}
then guarantees that $\Lambda_{n}$ is finitely presentable. 
\end{itemize}
In case $Q_{n}$ is trivial, $\tilde{Q}_{n}$ is also trivial, and
thus the fiber product over $\tilde{Q}_{n}$ is simply the full direct
product $\Lambda_{n}=\Gamma_{n}\times\Gamma_{n}$. 

One then take $\pi_{n}$ to be the abelianization of a presentation
of $\Gamma_{n}\times\Gamma_{n}$. 

By construction, when $b_{1}(\Lambda_{n})=r_{n}$, which is by \cite{Bridson2011}
when $Q_{n}$ is the trivial group, $\pi_{n}$ gives a presentation
of the abelianization of $\Lambda_{n}$.
\end{proof}
We can now prove Theorem \ref{thm:no uniform CFQ linear}.
\begin{proof}
[Proof of Theorem \ref{thm:no uniform CFQ linear}]Let $(m_{n},r_{n},S_{n},\pi_{n})$
be a sequence as given by Theorem \ref{thm:,-Corollary-D} together
with Proposition \ref{prop:additional prop}. 

Suppose that the groups $\Lambda_{n}:=\langle S_{n}\rangle$ have
uniformly computable finite quotients. 

We describe an algorithm which, on input $n$, decides whether $b_{1}(\Lambda_{n})=r_{n}$. 

The presentation $\pi_{n}$ defines a group $A_{n}$, which is the
abelianization of $\Lambda_{n}$ provided that $b_{1}(\Lambda_{n})=r_{n}$. 

By construction \cite{Bridson2011}, if $b_{1}(\Lambda_{n})\ne r_{n}$,
then $b_{1}(\Lambda_{n})>r_{n}$, and it is easy to find a finitely
generated abelian group which is not a quotient of $A_{n}$ but which
is a quotient of any abelian group of higher rank: this is true of
$B_{n}=(\mathbb{Z}/p\mathbb{Z})^{r_{n}+1}$, for $p$ a prime bigger
than the cardinality of the torsion part of $\Lambda_{n}$.

To prove that $b_{1}(\Lambda_{n})\ne r_{n}$, it thus suffices to
prove that $B_{n}$ is a quotient of $\Lambda_{n}$. 

Note that $B_{n}$ is finite, and thus it admits finitely many markings,
and thus from the hypothesis that the groups $\Lambda_{n}$ have uniformly
computable marked finite quotients, we can decide whether $B_{n}$
is a quotient of $\Lambda_{n}$, and thus whether $b_{1}(\Lambda_{n})=r_{n}$.

This is the desired contradiction.
\end{proof}

\section{Detecting linearity of subgroups of a finitely presented group with
solvable word problem}

Here we prove Proposition \ref{prop:fp wp linear subgroups}:
\begin{prop*}
[Proposition \ref{prop:fp wp linear subgroups}] There exists a
finitely presented group $G$ with solvable word problem such that
no algorithm can, on input a tuple of elements of $G$, stop if and
only if the group they generate is linear, nor can it stop if and
only if the group they generate is not linear.
\end{prop*}
This proposition relies on the following result of \cite{Rauzy2021}:
\begin{thm}
[\cite{Rauzy2021}, Theorem 9.4] \label{thm:Thm comp analysis gps}Suppose
that a $\Lambda_{WP}$-computable sequence $(G_{n})_{n\in\mathbb{N}}$
of $k$-marked groups effectively converges to a $k$-marked group
$H$, and suppose that $H\notin\{G_{n},n\in\mathbb{N}\}$. Then there
exists a finitely presented group $\Gamma$, with solvable word problem,
in which no algorithm can, given a tuple of elements of $\Gamma$
that defines a marked group of \textup{$\{G_{n},n\in\mathbb{N}\}\cup\{H\}$},
stop if and only if this tuple defines $H$. 
\end{thm}

Here, a sequence of marked groups is called $\Lambda_{WP}$\emph{-computable}
if the word problem can be solved uniformly in the sequence: there
is an algorithm which, on input $n$ and an element $w$ of the free
group $\mathbb{F}_{k}$, decides if $w=_{G_{n}}1$. 

This sequence \emph{effectively converges to} $H$ if there is an
algorithm which, on input $n$, produces $M$ such that for all $m\ge M$,
the ball of radius $n$ of the group $H$ coincides with that of $G_{m}$. 

Note that the set of finitely generated linear groups is $\Sigma_{2}^{0}$-complete
in the space of marked groups, in particular it is neither closed
nor open, below we use the fact that it is in some sense ``effectively
not closed'' and ``effectively not open''. 
\begin{proof}
[Proof of Proposition \ref{prop:fp wp linear subgroups}] It suffices
to use Theorem \ref{thm:Thm comp analysis gps} with both a $\Lambda_{WP}$\emph{-}computable
sequence of linear groups that converges to a non-linear group, and
a $\Lambda_{WP}$\emph{-}computable sequence of non-linear groups
that converges to a linear group. 

For instance, one can use: 
\begin{itemize}
\item A sequence of finite groups converging to a non-residually finite
group (thus non-linear). For instance, one could take the group of
permutations of $\mathbb{Z}$ generated by $n\mapsto n+1$ and $(1,2)$,
which is known to be LEF but not residually finite \cite{VershikGordon1998}
(and it is easy to have a computable sequence that effectively converges
witness for this). 
\item The free group $\mathbb{F}_{k}$ is the limit of the sequence of free
Burnside groups $\mathbb{F}_{k}/\mathbb{F}_{k}^{p}$ when the exponent
$p$ goes to infinity. Burnside groups are not linear. The sequence
is computable, since the Adian-Novikov construction solves the word
problem in these groups \cite{adian1979burnside}, and it effectively
converges. \qedhere
\end{itemize}
\end{proof}
\bibliographystyle{plain}
\bibliography{C:/Users/Emmanuel/Desktop/Articles/theonebib/TheOneBib}

\begin{thebibliography}{10}

\bibitem{adian1979burnside}
Sergei~I. Adian.
\newblock {\em The Burnside Problem and Identities in Groups}, volume~95 of
  {\em Ergebnisse der Mathematik und ihrer Grenzgebiete}.
\newblock Springer-Verlag, Berlin, Heidelberg, 1979.

\bibitem{Baumslag2000}
Gilbert Baumslag, Martin~R. Bridson, Charles~F. Miller, III, and Hamish Short.
\newblock Fibre products, non-positive curvature, and decision problems.
\newblock {\em Commentarii Mathematici Helvetici}, 75(3):457--477, 2000.

\bibitem{breuillard2024undecidabilityMatrix}
Emmanuel Breuillard and Georgi Kocharyan.
\newblock Undecidability of the stabilizer and zero-in-the-corner problems for
  matrix groups, 2024.

\bibitem{Bridson2011}
Martin~R. Bridson and Henry Wilton.
\newblock On the difficulty of presenting finitely presentable groups.
\newblock {\em Groups, Geometry, and Dynamics}, pages 301--325, 2011.

\bibitem{CASSAIGNE1999}
Julien Cassaigne, Tero Harju, and Juhani Karhum{\"a}ki.
\newblock On the undecidability of freeness of matrix semigroups.
\newblock {\em International Journal of Algebra and Computation},
  09(03n04):295--305, 1999.

\bibitem{collins_miller_1999}
Donald~J. Collins and Charles~F. Miller, III.
\newblock The word problem in groups of cohomological dimension 2.
\newblock In C.~M. Campbell, E.~F. Robertson, N.~Ruskuc, and G.~C. Smith,
  editors, {\em Groups St Andrews 1997 in Bath: Volume 1}, number 260 in London
  Mathematical Society Lecture Note Series, pages 211--218. Cambridge
  University Press, 1999.

\bibitem{Detinko_Eick_Flannery_2011}
A.~S. Detinko, B.~Eick, and D.~L. Flannery.
\newblock {\em Computing with matrix groups over infinite fields}, pages
  256--270.
\newblock London Mathematical Society Lecture Note Series. Cambridge University
  Press, 2011.

\bibitem{Detinko2019}
A.S. Detinko and D.L. Flannery.
\newblock Linear groups and computation.
\newblock {\em Expositiones Mathematicae}, 37(4):454--484, dec 2019.

\bibitem{DETINKO2013100}
A.S. Detinko, D.L. Flannery, and E.A. O'Brien.
\newblock Recognizing finite matrix groups over infinite fields.
\newblock {\em Journal of Symbolic Computation}, 50:100--109, 2013.

\bibitem{Haglund2008}
Fr{\'e}d{\'e}ric Haglund and Daniel~T. Wise.
\newblock Special cube complexes.
\newblock {\em Geometric and Functional Analysis}, 17(5):1551--1620, 2008.

\bibitem{Kharlampovich2017}
Olga Kharlampovich, Alexei Myasnikov, and Mark Sapir.
\newblock Algorithmically complex residually finite groups.
\newblock {\em Bulletin of Mathematical Sciences}, 7(2):309--352, mar 2017.

\bibitem{Harlampovic1982}
Olga~G. Kharlampovich.
\newblock A finitely presented solvable group with unsolvable word problem.
\newblock {\em Mathematics of the {USSR}-Izvestiya}, 19(1):151--169, feb 1982.

\bibitem{Malcev1940}
Anatolii~I. Malcev.
\newblock On isomorphic matrix representations of infinite groups.
\newblock {\em Rec. Math. [Mat. Sbornik] N.S.}, 8 (50)(3):405--422, 1940.

\bibitem{Malcev1956ResFin}
Anatolii~I. Malcev.
\newblock On homomorphisms onto finite groups.
\newblock {\em Uchenye Zapiski Ivanovskogo Gosudarstvennogo Pedagogicheskogo
  Instituta}, 18:49--60, 1956.

\bibitem{Malcev1983ResFin}
Anatolii~I. Malcev.
\newblock On homomorphisms onto finite groups.
\newblock {\em American Mathematical Society Translations: Series 2},
  119:67--79, 1983.
\newblock Translation of Uchen. Zap. Ivanov. Gos. Ped. Inst. 18 (1956), 49--60.

\bibitem{Mihailova1968Occurrence}
K.~A. Mihailova.
\newblock The occurrence problem for free products of groups.
\newblock {\em Mathematics of the {USSR}-Sbornik}, 4(2):181--190, feb 1968.

\bibitem{Miller1992}
Charles~F. {Miller III}.
\newblock Decision problems for groups {\textemdash} survey and reflections.
\newblock In {\em Mathematical Sciences Research Institute Publications}, pages
  1--59. Springer New York, 1992.

\bibitem{OuldHoucine2007}
Abderezak Ould~Houcine.
\newblock Embeddings in finitely presented groups which preserve the center.
\newblock {\em Journal of Algebra}, 307(1):1--23, January 2007.

\bibitem{Rabin1960}
Michael~O. Rabin.
\newblock Computable algebra, general theory and theory of computable fields.
\newblock {\em Transactions of the American Mathematical Society}, 95(2):341,
  may 1960.

\bibitem{Rauzy2020}
Emmanuel Rauzy.
\newblock Obstruction to a {H}igman embedding theorem for residually finite
  groups with solvable word problem.
\newblock {\em Journal of Group Theory}, 24(3):445--452, oct 2020.

\bibitem{Rauzy2021}
Emmanuel Rauzy.
\newblock Computable analysis on the space of marked groups.
\newblock {\em arXiv:2111.01179}, 2021.

\bibitem{Rauzy_2021}
Emmanuel Rauzy.
\newblock Computability of finite quotients of finitely generated groups.
\newblock {\em Journal of Group Theory}, 25(2):217--246, 2022.

\bibitem{Sapir2018Square}
Mark Sapir.
\newblock Answer to ``for which kinds of group $g$, can we identify a square
  element efficiently?''.
\newblock MathOverflow answer to question 300063.

\bibitem{VershikGordon1998}
Anatolii~M. Vershik and Evgenii~I. Gordon.
\newblock Groups that are locally embeddable in the class of finite groups.
\newblock {\em St. Petersburg Math.}, 1998.

\end{thebibliography}

\end{document}